\documentclass[11pt,a4paper]{article}
\usepackage{amsmath}
\usepackage{amsthm}
\usepackage{amssymb}
\usepackage{amscd}
\usepackage{graphicx}
\usepackage{epsfig}
\usepackage[matrix,arrow,curve]{xy}

\usepackage{latexsym}
\usepackage{amsfonts}
\input xy
\usepackage{tikz}
\usetikzlibrary{arrows,chains,matrix,positioning,scopes,snakes}
\makeatletter
\tikzset{join/.code=\tikzset{after node path={%
\ifx\tikzchainprevious\pgfutil@empty\else(\tikzchainprevious)%
edge[every join]#1(\tikzchaincurrent)\fi}}}
\makeatother

\tikzset{>=stealth',every on chain/.append style={join},
         every join/.style={->}}
\tikzset{
    >=stealth',
    punkt/.style={
           rectangle,
           rounded corners,
           draw=black, very thick,
           text width=6.5em,
           minimum height=2em,
           text centered},
    pil/.style={
           ->,
           thick,
           shorten <=2pt,
           shorten >=2pt,}
}

\xyoption{all} \tolerance=500

\newtheorem{thm}{Theorem}[section]
\newtheorem{prop}[thm]{Proposition}
\newtheorem{lem}[thm]{Lemma}

\theoremstyle{definition}

\newtheorem{remark}[thm]{Remark}
\newtheorem{definition}[thm]{Definition}

\font\black=cmbx10 \font\sblack=cmbx7 \font\ssblack=cmbx5 \font\blackital=cmmib10  \skewchar\blackital='177
\font\sblackital=cmmib7 \skewchar\sblackital='177 \font\ssblackital=cmmib5 \skewchar\ssblackital='177
\font\sanss=cmss12 \font\ssanss=cmss8 scaled 900 \font\sssanss=cmss8 scaled 600 \font\blackboard=msbm10
\font\sblackboard=msbm7 \font\ssblackboard=msbm5 \font\caligr=eusm10 \font\scaligr=eusm7 \font\sscaligr=eusm5

\font\bsymb=cmsy10 scaled\magstep2
\def\all#1{\setbox0=\hbox{\lower1.5pt\hbox{\bsymb
       \char"38}}\setbox1=\hbox{$_{#1}$} \box0\lower2pt\box1\;}
\def\exi#1{\setbox0=\hbox{\lower1.5pt\hbox{\bsymb \char"39}}
       \setbox1=\hbox{$_{#1}$} \box0\lower2pt\box1\;}

\newfam\bifam
\textfont\bifam=\blackital \scriptfont\bifam=\sblackital \scriptscriptfont\bifam=\ssblackital

\newfam\blfam
\textfont\blfam=\black \scriptfont\blfam=\sblack \scriptscriptfont\blfam=\ssblack

\newfam\bbfam
\textfont\bbfam=\blackboard \scriptfont\bbfam=\sblackboard \scriptscriptfont\bbfam=\ssblackboard

\newfam\ssfam
\textfont\ssfam=\sanss \scriptfont\ssfam=\ssanss \scriptscriptfont\ssfam=\sssanss

\newfam\clfam
\textfont\clfam=\caligr \scriptfont\clfam=\scaligr \scriptscriptfont\clfam=\sscaligr

\def\pmb#1{\setbox0\hbox{${#1}$} \copy0 \kern-\wd0 \kern.2pt \box0}
\def\pmbb#1{\setbox0\hbox{${#1}$} \copy0 \kern-\wd0
      \kern.2pt \copy0 \kern-\wd0 \kern.2pt \box0}
\def\pmbbb#1{\setbox0\hbox{${#1}$} \copy0 \kern-\wd0
      \kern.2pt \copy0 \kern-\wd0 \kern.2pt
    \copy0 \kern-\wd0 \kern.2pt \box0}
\def\pmxb#1{\setbox0\hbox{${#1}$} \copy0 \kern-\wd0
      \kern.2pt \copy0 \kern-\wd0 \kern.2pt
      \copy0 \kern-\wd0 \kern.2pt \copy0 \kern-\wd0 \kern.2pt \box0}
\def\pmxbb#1{\setbox0\hbox{${#1}$} \copy0 \kern-\wd0 \kern.2pt
      \copy0 \kern-\wd0 \kern.2pt
      \copy0 \kern-\wd0 \kern.2pt \copy0 \kern-\wd0 \kern.2pt
      \copy0 \kern-\wd0 \kern.2pt \box0}

\mathchardef\za="710B  
\mathchardef\zb="710C  
\mathchardef\zg="710D  
\mathchardef\zd="710E  
\mathchardef\zve="710F 
\mathchardef\zz="7110  
\mathchardef\zh="7111  
\mathchardef\zvy="7112 
\mathchardef\zi="7113  
\mathchardef\zk="7114  
\mathchardef\zl="7115  
\mathchardef\zm="7116  
\mathchardef\zn="7117  
\mathchardef\zx="7118  
\mathchardef\zp="7119  
\mathchardef\zr="711A  
\mathchardef\zs="711B  
\mathchardef\zt="711C  
\mathchardef\zu="711D  
\mathchardef\zvf="711E 
\mathchardef\zq="711F  
\mathchardef\zc="7120  
\mathchardef\zw="7121  
\mathchardef\ze="7122  
\mathchardef\zy="7123  
\mathchardef\zf="7124  
\mathchardef\zvr="7125 
\mathchardef\zvs="7126 
\mathchardef\zf="7127  
\mathchardef\zG="7000  
\mathchardef\zD="7001  
\mathchardef\zY="7002  
\mathchardef\zL="7003  
\mathchardef\zX="7004  
\mathchardef\zP="7005  
\mathchardef\zS="7006  
\mathchardef\zU="7007  
\mathchardef\zF="7008  
\mathchardef\zW="700A  

\newcommand{\be}{\begin{equation}}
\newcommand{\ee}{\end{equation}}
\newcommand{\raa}{\rightarrow}

\newcommand{\bea}{\begin{eqnarray}}
\newcommand{\eea}{\end{eqnarray}}
\newcommand{\beas}{\begin{eqnarray*}}
\newcommand{\eeas}{\end{eqnarray*}}
\def\*{{\textstyle *}}

\newcommand{\w}{\wedge}

\def\cA{{\cal A}}

\def\cE{{\cal E}}

\def\cN{{\cal N}}

\def\cJ{{\cal J}}

\def\cO{{\cal O}}

\def\cM{{\cal M}}

\newdir{|>}{%
!/4.5pt/@{|}*:(1,-.2)@^{>}*:(1,+.2)@_{>}}

\newcommand{\la}{\langle}
\newcommand{\ra}{\rangle}

\newcommand{\Ci}{C^{\infty}}

\newcommand{\N}{\mathbb{N}}
\newcommand{\Z}{\mathbb{Z}}
\newcommand{\R}{\mathbb{R}}

\newcommand{\lp}{\left(}
\newcommand{\rp}{\right)}

\newcommand{\op}[1]{\!\!\mathop{\rm ~#1}\nolimits}

\newcommand{\id}{\op{id}}

\begin{document}
\title{\bf $\mathbb{Z}_2^n$-Supergeometry II\\ Batchelor-Gawedzki Theorem}
\date{}
\author{Tiffany Covolo, Janusz Grabowski, and Norbert Poncin}

\maketitle

\begin{abstract} {Quite a number of $\Z_2^n$-gradings, $n\ge 2$, appear in Physics and in Mathematics. The corresponding sign rules are given by the `scalar product' of the involved $\Z_2^n$-degrees.  The new theory exhibits challenging differences with the classical one: nonzero degree even coordinates are not nilpotent, and even (resp., odd) coordinates do not necessarily commute (resp., anticommute) pairwise (the parity is the parity of the total degree). Formal series are the appropriate substitute for nilpotency; the category of $\Z_2^\bullet$-manifolds is closed with respect to the tangent and cotangent functors. The $\Z_2^n$-supergeometric viewpoint provides deeper insight and simplified solutions; interesting relations with Quantum Field Theory and Quantum Mechanics are expected. In this article, we introduce split $\Z_2^n$-manifolds as intrinsic superizations of $\Z_2^n\setminus\{0\}$-graded vector bundles and prove that, conversely, any $\Z_2^n$-manifold is noncanonically split. We thus provide a complete proof of the $\Z_2^n$-extension of the so-called Batchelor-Gawedzki Theorem.}
\end{abstract}

\vspace{2mm} \noindent {\bf MSC 2010}: 17A70, 58A50, 13F25, 16L30 \medskip


\thispagestyle{empty}
\tableofcontents

\section{Introduction}

This paper is the second of a series of articles on $\Z_2^n$-Supergeometry. For examples of $\Z_2^n$-superalgebras and $\Z_2^n$-supermanifolds, for motivations, the discussion of Neklyudova's equivalence \cite{Lei11} and the necessity and sufficiency of $\Z_2^n$-gradings, for a detailed study of $\Z_2^n$-supermanifolds and their morphisms, as well as for expected applications in Physics, we refer the reader to \cite{CGP14}.\medskip

{Let us nevertheless emphasize that $\Z_2^n$-Supergeometry is a `necessary' and `sufficient' generalization. Indeed, if $\la -,-\ra$ denotes the standard `scalar product' of $\Z_2^n$, the considered $\Z_2^n$-commutative algebras are sufficient, in the sense that any (!) sign rule, for any (!) finite number $m$ of parameters $\xi^a$, is of the form \be\label{Z2n-com} \xi^a\xi^b = (-1)^{\la \zs_a\zs_b\ra} \xi^b\xi^a\;,\ee for some $\Z_2^n$-degree $\zs:\{\xi^1,\ldots,\xi^m\}\ni \xi^a\mapsto \zs_a\in\Z_2^n$ and some $n\le 2m$ \cite{CGP14}. On the other hand, such $\Z_2^n$-commutative algebras appear naturally in standard Supergeometry. It suffices to think about the two possible function sheaves of the tangent bundle of a classical supermanifold $\cM$, i.e. about the sheaf $$\zW^\bullet_{\op{BL}}(\cM)=\bigoplus_k\zW^k_{\op{BL}}(\cM)\;$$ $$\label{DeligneDiffForms}(\;\text{resp.,}\quad\;\zW^\bullet_{\op{D}}(\cM)=\bigoplus_k\zW^k_{\op{D}}(\cM)\;)$$ of differential superforms with commutation rule $$\zw \zw'=(-1)^{(\tilde\zw+k)(\tilde\zw'+\ell)}\zw'\zw$$ $$(\;\text{resp., }\;\zw\zw'=(-1)^{\tilde\zw\tilde\zw'+k\ell}\zw'\zw\;)\;,$$ where $\zw$ (resp., $\zw'$) is a form of $\N$-degree $k$ and parity $\tilde\zw$ (resp., of $\N$-degree $\ell$ and parity $\tilde\zw'$).}\medskip

{It is worth recalling that the first choice (which corresponds to a de Rham differential of parity 1) turns the tangent bundle into a classical supermanifold $T[1]\cM$, provided we complete the sheaf $\zW^\bullet_{\op{BL}}(\cM)$ of BL-differential forms by the sheaf $\widehat{\zW}_{\op{BL}}(\cM)$ of pseudodifferential forms (which does in general not carry any $\N$-grading) \cite{Lei11}. For the second choice (de Rham differential of parity 0) the tangent bundle of $\cM$ becomes a $\Z_2^2$-supermanifold $T\cM$, if we consider the completion $$\widehat{\zW}_{\op{D}}(\cM):=\prod_k\zW^k_{\op{D}}(\cM)$$ of the sheaf $\zW^\bullet_{\op{D}}(\cM)$ of $\op{D}$-differential forms. In the first case, local coordinates $$(x^i,\xi^a,\dot{x}^j\simeq dx^j,\dot{\xi}^b\simeq d\xi^b)$$ have the parities $$(0,1,1,0)\;,$$ whereas they have the bidegrees $$((0,0),(0,1),(1,0),(1,1))$$ in the second. Locally, the completed algebra of superfunctions of the supermanifold $T[1]\cM$ is of the form \be\label{Compl}\Ci(U)[\xi^a,dx^j]\;,\ee where $U$ is an open subset of a Euclidean space with coordinates $(x^i,d\xi^b)$; on the other hand, the completed algebra of $\Z_2^2$-functions of the $\Z_2^2$-manifold $T\cM$ is locally of the type \be\label{Ser}\Ci(V)[[\xi^a,dx^j,d\xi^b]]\;,\ee i.e., it is made of the formal power series -- in the considered $\Z_2^2$-commutative (see (\ref{Z2n-com})) indeterminates of non-zero degree -- with coefficients in the smooth functions of an open subset $V$ of a Euclidean space with coordinates $x^i$. A further discussion and additional motivations of these different viewpoints (\ref{Compl}) and (\ref{Ser}), can be found in \cite{CGP14}, Section 1, Paragraph 4, and Section 3.2.}\medskip

The present paper on the $\Z_2^n$-extension of the Batchelor-Gawedzki Theorem is organized as follows.\medskip

For the reader's convenience, we briefly summarize in Section 2 the main results of \cite{CGP14} that we use in this text.

In Section 3, we recall that any classical supermanifold is noncanonically the superization $$\zP E=E[1]$$ of a vector bundle $E$ ($[1]$ means that parity $1$ is assigned to all fiber coordinates), and that any $\N$-manifold of order $n\ge 1$ is noncanonically the superization \be\label{NMfdsGVB}\zP E=\bigoplus_{i=1}^nE_{-i}[i]\ee of a graded vector bundle $E=\bigoplus_{i=1}^nE_{-i}$ concentrated in degrees $-1,\ldots,-n$ (again $[i]$ means that the degree of the fiber coordinates is $i$). We then explain that the superization $$\zP E=\bigoplus_{\zs_i\in\Z_2^n\setminus\{0\}}E_{\zs_i}[\zs_i]$$ of a $\Z_2^n\setminus\{0\}$-graded vector bundle $E$ leads to a $\Z_2^n$-supermanifold with function sheaf $$\cA(\zP E)=\prod_{k\ge 0}\zG(\odot^k(\zP E)^*)\;,$$ where $\odot$ denotes the $\Z_2^n$-graded symmetric tensor product (the passage to opposite degrees, see (\ref{NMfdsGVB}), is redundant here since the degrees are considered `modulo 2'). Moreover, we stress the difference between the {\it split} $\Z_2^n$-supermanifolds $\zP E$ and the $\Z_2^n$-supermanifolds obtained via superization of $n$-vector bundles, see \cite{CGP14}, Example 5.4. The latter are examples of usually not canonically split $\Z_2^n$-supermanifolds.

The main result of this work, the $\Z_2^n$-Batchelor-Gawedzki Theorem, can be found in Section 4: any smooth $\Z_2^n$-supermanifold is noncanonically split.\smallskip

{In the literature, this type of results is usually referred to as Batchelor theorems. Indeed, for $n=1$, a proof of the preceding statement can be found in \cite{Bat1}, \cite{Bat2}. However, a variant of this splitting theorem for smooth $\Z_2$-supermanifolds had already been proved a bit earlier by Gawedzki \cite{Gaw77}. Moreover, D. Leites informed us that A. A. Kirillov and A. N. Rudakov convinced themselves independently of the correctness of this claim while using an elevator at Moscow State University. It is known that the Batchelor-Gawedzki theorem for $\Z_2$-supermanifolds holds not only in the smooth, but also in the real analytic category \cite{Vis11}, \cite{Vis14}. The situation is different for complex analytic $\Z_2$-supermanifolds: there exist holomorphic $\Z_2$-supermanifolds whose structure sheaf is NOT isomorphic to the sheaf of sections of a bundle of exterior algebras \cite{Gre82}.\smallskip}

Let us come back to the $\Z_2^n$-Batchelor-Gawedzki Theorem. Its proof contains 3 steps.

1. For any $\Z_2^n$-supermanifold $\cM=(M,\cA)$, we have a short exact sequence of sheaves 
\be\label{BaSeq} 0\to\cJ\to\cA\to\Ci\to 0\;,\ee 
and $\cJ/\cJ^2\simeq\zG((\zP E)^*)$, where $E$ is a $\Z_2^n\setminus\{0\}$-graded vector bundle. The sheaves $\cA(\zP E)$ and $\cA=\cA(\cM)$ are locally isomorphic and their global difference comes from a cohomological invariant.

2. To build a sheaf morphism $\cA(\zP E)\to \cA$, we need a projection $\cM\to M$, or, more precisely, a splitting $\zf$ of (\ref{BaSeq}). The construction of the (noncanonical) latter is the most challenging part of the proof.

3. This embedding $\zf:\Ci\to \cA$ implies that $$0\to\cJ^2\to\cJ\to\cJ/\cJ^2\to 0$$ is a short exact sequence of sheaves of $\Ci$-modules. Although $\cJ$ and $\cJ^2$ are not locally free, we can obtain a splitting $$\Phi:\cJ/\cJ^2\simeq \zG((\zP E)^*)\to \cJ\subset\cA$$ and show that it extends to an `algebra' morphism $$\Phi:\cA(\zP E)=\prod_{k\ge 0}\zG(\odot^k (\zP E)^*)\to \cA\;.$$

Let us finally provide a non-exhaustive list of references on classical supermanifolds and related topics that were of importance for the present text: \cite{DAL}, \cite{Lei11}, \cite{Var}, \cite{Man}, \cite{DM99}, \cite{CCF}, \cite{DSB}, \cite{Vor12}, \cite{CR12}, \cite{BP12}, \cite{GKP1}, \cite{GKP2}.

\section{Preliminaries}

In principle, in the following we freely use notation, definitions, and the results of \cite{CGP14}. For the convenience of the reader, we nevertheless recall some definitions and propositions in the present section. For explanations on adic topologies, ringed spaces and their morphisms, sheaves, as well as partitions of unity, we refer the reader to the Appendix of \cite{CGP14}.

\begin{definition} A \emph{locally $\Z_2^n$-ringed space} ({\small LZRS}), $n\in\N\setminus\{0\}$, is a pair $(M, \cA_M)$ made of a second-countable Hausdorff space $M$ and a sheaf $\cA_M$ of $\Z_2^n$-graded $\Z_2^n$-commutative associative unital $\R$-algebras, such that the stalks $\cA_{m}$, $m\in M$, be local rings.\end{definition}

In this definition, \emph{$\Z_2^n$-commutative} means that two sections $s,t\in\cA_M(U)$, $U\subset M$ open, of $\Z_2^n$-degree $\tilde s,\tilde t,$ commute according to the sign rule \be ts=(-1)^{\la \tilde s,\tilde t\ra}st\;,\ee where $\la-,-\ra$ is the standard `scalar product' of $\Z_2^n.$

\begin{definition}\label{DefZSupMan} Let $n\in\N\setminus\{0\}$, $p\in\N$, and $\mathbf{q}=(q_1,\ldots,q_{2^n-1})\in \N^{2^n-1}$. A smooth \emph{$\Z_2^n$-supermanifold} $\cM$ of dimension $p|\mathbf{q}$ is a {\small (L)ZRS} $\cM=(M,\cA_M)$ that is locally isomorphic to a $\Z_2^n$-superdomain ${\cal U}^{p|\mathbf{q}}$. By $\Z_2^n$-superdomain we mean a {\small LZRS} of the type $${\cal U}^{p|\mathbf{q}}=(U,\Ci_U[[\xi^1,\ldots,\xi^q]])\;,$$ where $U\subset \R^p$ is open, the brackets $[[\ldots]]$ denote formal power series, where $q=|\mathbf{q}|:=\sum_iq_i$, and where $\xi^1,\ldots,\xi^q$ are formal variables of which $q_k$ have the $k$-th degree in $\Z_2^n\setminus\{0\}$, $0=(0,\ldots,0)$, endowed with the lexicographical order. \end{definition}

Note that the sections in $\cO_U(V):=\Ci_U(V)[[\xi^1,\ldots,\xi^q]]$, $V$ open in $U$, are the formal series $$\sum_{|\zm|\ge0}f_\zm(x)\xi^\zm=\sum_{\zm_1+\ldots+\zm_q\ge 0}f_{\zm_1\ldots\zm_q}(x)\,(\xi^1)^{\zm_1}\ldots(\xi^q)^{\zm_q}\;.$$ Here $\zm$ is a multi-index $\zm\in\N^q$, with $\zm_a\in\{0,1\}$ if $\xi^a$ is nilpotent, and $x=(x^1,\ldots,x^p)$ are the coordinates in $V$.

\begin{prop} The topological base space $M$ of a smooth $\Z_2^n$-supermanifold $\cM=(M,\cA_M)$ of dimension $p|\mathbf{q}$ carries a classical smooth manifold structure of dimension $p$, and there exists a short exact sequence of sheaves 
$$\label{BasicSeq} 0\to {\cal J}_M\to \cA_M\stackrel{\ze}{\to }\Ci_M\to 0\;.$$\end{prop}

\begin{prop} The function sheaf $\cA_M$ of a $\Z_2^n$-supermanifold $\cM=(M,\cA_M)$ is Hausdorff-complete with respect to the $\cJ_M$-adic topology: \be\cA_M=\varprojlim_k\cA_M/\cJ_M^k\;.\ee This property also holds for the algebras of sections $\cA(U)$ and their ideals $\cJ(U)$, $U$ open in $M$. \end{prop}

Let us also recall that, due to the existence of partitions of unity for $\Z_2^n$-supermanifolds, the presheaves $\cA_M/\cJ_M^k$, $k\ge 1$, are in fact sheaves.

\begin{prop} For any $\Z_2^n$-supermanifold ${\cal M} =(M,\cA_M)$ and any point $m\in M$, the unique maximal homogeneous ideal ${\frak m}_m$ of the stalk $\cA_m$ is given by \be\label{MaxIdea}{\frak m}_m=\{[f]_m:(\ze f)(m)=0\}\;.\ee\end{prop}

When taking an interest in the stalks $\cA_m$ of the function sheaf of a $\Z_2^n$-supermanifold $(M,\cA_M)$ of dimension $p|\mathbf{q}$, we can choose a centered chart $(x,\xi)=(x^1,\ldots,x^p,\xi^1,\ldots,\xi^q)$ around $m$ and work in a $\Z_2^n$-superdomain ${\cal U}^{p|\mathbf{q}}$ associated with a convex open subset $U\subset \R^p,$ in which $m\simeq x=0$. In view of (\ref{MaxIdea}), a Taylor expansion (with remainder) around $m\simeq x=0$ of the coordinate form of $\ze f$ shows that $${\frak m}_m\simeq \{[f]_0: f(x,\xi) = 0(x) + \sum_{|\zm|>0}f_\zm(x)\xi^\zm \}\;,$$ where $0(x)$ are terms of degree 1 at least in $x$.

\begin{prop}\label{Claim4} For any $m\in M$, the basis ${\frak m}_m^{k+1}$ $(k\ge 0)$ of neighborhoods of $\,0$ in the ${\frak m}_m$-adic topology of $\cA_m$ is given by 
$$
{\frak m}_m^{k+1}=\{[f]_0: f(x,\xi) = \sum_{0\le|\zm|\le k}0_\zm(x^{k-|\zm|+1})\xi^\zm+\sum_{|\zm|>k}f_\zm(x)\xi^\zm\}\;,\label{lc2}
$$ 
where notation is the same as above. \end{prop}

A morphism of $\Z_2^n$-supermanifolds or $\Z_2^n$-morphism is a morphism between the underlying locally $\Z_2^n$-ringed spaces.

\begin{prop}\label{Claim1} Any $\Z_2^n$-morphism $\Psi=(\psi,\psi^*):\cM=(M,\cA_M)\to \cN=(N,{\cal B}_N)$ is continuous with respect to $\cal J$ and $\frak m$, i.e., for any open $V\subset N$ and any $m\in M$, we have
$$ \psi^{*}_V\lp \cJ_N(V) \rp \subset \cJ_M({\psi^{-1}(V)})\;\;\text{and}\;\; \psi^{*}_{m}\lp \frak{m}_{\psi(m)}\rp \subset \frak{m}_{m}\;.\label{stalkcond}$$
\end{prop}

\begin{thm}\label{Claim3}
Let $m\in M$ be a base point of a $\Z_2^n$-supermanifold $\cM=(M,\cA_M)$ and let $f\in\cA_M(U)$ be a $\Z_2^n$-function defined in a neighborhood $U$ of $m$. For any fixed degree of approximation $k\in\N\setminus\{0\}$, there exists a polynomial $P=P(x,\xi)$ such that $$[f]_{m}-[P]_{m} \in \frak{m}_{m}^{k}\;.$$
\end{thm}

In this statement the polynomial $P$ depends on $m,$ $f$, and $k$, and the variables $(x,\xi)$ are (pullbacks of) coordinates centered at $m$. Let us further emphasize that here and in the following, the term `polynomial section' refers to a formal series $\sum_{|\zm|\ge 0} P_\zm(x)\xi^\zm$ in the parameters $\xi^a$ with coefficients $P_\zm(x)\in \op{Pol}_V(V)$ that are polynomial in the base variables $x^i$.

\begin{thm}\label{FundaTheoSuperm} If ${\cal M}=(M,\cA_M)$ is a $\Z_2^n$-supermanifold of dimension $p|\mathbf{q}\,$, $${\cal V}^{\,u|\mathbf{v}}=(V,\Ci_V[[\xi^1,\ldots,\xi^v]])$$ a $\Z_2^n$-superdomain of dimension $u|\mathbf{v}$, $v=|\mathbf{v}|$, and if $(s^j,\zeta^b)$ is an $(u+v)$-tuple of homogeneous $\Z_2^n$-functions in $\cA_M(M)$ that have the same $\Z_2^n$-degrees as the coordinates $(x^j,\xi^b)$ in ${\cal V}^{\,u|\mathbf{v}}$ and satisfy $\lp  \ze s^{1}, \ldots, \ze s^{u}\rp(M)\subset V,$ there exists a unique morphism of $\Z_2^n$-supermanifolds
$\Psi =(\psi, \psi^{*}) :  {\cal M} \raa {\cal V}^{u|\mathbf{v}},$ such that $ s^{j}= \psi^{*}_Vx^{j}$ and $\zeta^{b}= \psi^{*}_V\xi^{b}$.\end{thm}

\section{Split $\Z_2^n$-supermanifolds}

The prototypical (smooth) supermanifold is the locally super ringed space ({\small LSRS}) $(M,\zG(\w T^*M))$ of differential forms over a classical (smooth) manifold $M$. More generally, if $E$ is a vector bundle over $M$, the {\small LSRS} $(M,\zG(\w E^*))$ is a supermanifold of dimension $p|q$, where $p=\dim M$ and $q=\op{rank}E$. This supermanifold $(M,\zG(\w E^*))$ is usually denoted $\zP E$ or $E[1]$ and viewed as the total space of the vector bundle $E$ with fiber coordinates of parity $1$. We refer to a supermanifold $\zP E=E[1]$ induced by a vector bundle as a \emph{split supermanifold}. The importance of this example relies on the fact that any smooth supermanifold is of this type \cite{Bat1}, \cite{Bat2}, \cite{Gaw77}. More precisely, for any smooth supermanifold $\cM=(M,\cA)$ over a classical smooth manifold $M$, there exists a vector bundle $E$ over $M$, such that $\cal M$ is diffeomorphic to $\zP E$. This isomorphism is noncanonical and cannot be used in the complex category. It is known as the \emph{Batchelor-Gawedzki Theorem}.\medskip

A similar proposition holds for $\N$-manifolds: Any $\N$-manifold $\cM=(M,\cA)$ of degree $n$, $n\in \N\setminus\{0\}$, is noncanonically diffeomorphic to a split $\N$-manifold $\zP E$, where $E=\bigoplus_{i=1}^nE_{-i}$ is a graded vector bundle over $M$ concentrated in degrees $-1,\ldots,-n$, and where $\zP E=\bigoplus_{i=1}^nE_{-i}[i]$ means that the fiber coordinates of $E_{-i}$ are viewed as having degree $i$ \cite{BP12}.\medskip

As already mentioned, the major objective of this paper is to extend the Batchelor-Gawedzki Theorem to $\Z_2^n$-supermanifolds. We first show that any $\Z_2^n\setminus\{0\}$-graded vector bundle $E$ implements a $\Z_2^n$-supermanifold $\zP E$. Begin, to simplify notation, with a $\Z_2^2\setminus\{0\}$-graded vector bundle $E=E_{01}\oplus E_{10}\oplus E_{11}$ over a manifold $M$, and set 
$$\zP E=E_{01}[01]\oplus E_{10}[10]\oplus E_{11}[11]\;,$$ 
where the degrees in the square brackets are assigned to the fiber coordinates. In view of the form of the coordinate transformations in the vector bundles $E_{ij}[ij]$, this assignment is consistent. Denote by 
$$(\zP E)^*=E_{01}[01]^*\oplus E_{10}[10]^*\oplus E_{11}[11]^*\;$$ 
the dual bundle -- the vectors of each bundle $E_{ij}[ij]^*$ have degree $ij$ -- and by $\odot^k (\zP E)^*$, $k\ge 2,$ the $\Z_2^n$-graded symmetric $k$-tensor bundle of $(\zP E)^*.$ 
By graded symmetric we mean here that, if $e'\in E_{ij; m}[ij]^*$ and $e''\in E_{k\ell; m}[k\ell]^*$, $m\in M,$ then 
$$e'\odot e''=(-1)^{ik+j\ell}e''\odot e'\;.$$ 
Consider now the function sheaf 
\be\label{GradMfdDVB}
\cA(\zP E):=\prod_{k\ge 0}\zG(\odot^k(\zP E)^*)\simeq\prod_{k\ge 0}\bigoplus_{r+s+t=k}\zG(\w^r E^*_{01}\otimes \w^s E^*_{10}\otimes \vee^t E^*_{11})\;,
\ee 
where $\w$ and $\vee$ denote the antisymmetric and symmetric tensor products, respectively. Of course, 
$$\zG(\w^r E^*_{01}\otimes \w^s E^*_{10}\otimes \vee^t E^*_{11})
=
\w^r \zG(E^*_{01})\otimes \w^s \zG(E^*_{10})\otimes {\vee}^t \zG(E^*_{11})\;,$$ 
where the {\small RHS} tensor products are over $\Ci$. The limit $\cA(\zP E)$ is a sheaf of $\Z_2^n$-graded $\Ci$-modules for the standard sum and action by scalars. Moreover, $\cA(\zP E)$ is a sheaf of $\Z_2^n$-commutative associative unital $\R$-algebras. 
The multiplication $\odot$ is also the standard one: when writing formal series $\sum_{k=0}^\infty\Psi_k$ instead of families 
$(\Psi_0,\Psi_1,\ldots)$, we get 
$$\label{MultFormSer}\sum_{k}\Psi_k'\odot\sum_\ell\Psi_\ell''=\sum_n\sum_{k+\ell=n}\Psi_k'\odot\Psi_\ell''\;.$$

\begin{remark} To understand how the $\Z_2^n$-commutativity of two vectors $v\in \zG(E_{ij}^*)$ and $w\in \zG(E_{k\ell}^*)$ is encoded in the {\small RHS} of (\ref{GradMfdDVB}), consider the following simplified situation. If $V$ and $W$ are real vector spaces, we have 
$$\label{TensSum}
\w(V\oplus W)\simeq \w V\otimes \w W\quad\text{and}\quad \w^n(V\oplus W)\simeq \bigoplus_{i+j=n}\w^i V\otimes \w^j W\;.
$$ 
For $n=2$, the isomorphism identifies 
$$(v+w)\w(v'+w')=v\w v'+v\w w'+w\w v'+w\w w'=v\w v'+v\w w'-v'\w w+w\w w'$$ with $$(v\w v')\otimes 1+v\otimes w'-v'\otimes w+1\otimes(w\w w')\;,$$ 
so that the anticommutation of elements $v'\in V\oplus W$ and $w\in V\oplus W$ seems to be lost. However, when defining the algebra structure on the {\small RHS}, we pull back the algebra structure from the {\small LHS}, i.e. we set $$(v\otimes w)\cdot (v'\otimes w')\simeq v\w w\w v'\w w'=-v\w v'\w w\w w'\simeq  -(v\w v')\otimes(w\w w')\;.$$\end{remark}\medskip

Let us now come back to the sheaf $\cA(\zP E)$. If we work over a common local trivialization domain $U\subset M$ of $E_{01}$, $E_{10}$, and $E_{11}$, and denote the base coordinates by $x$ and the fiber coordinates by $\xi$, $\zh$, and $\zy$, respectively ($\xi$, $\zh$, and $\zy$ are then also the base vectors of the fibers of $E_{01}^*$, $E_{10}^*$, and $E_{11}^*$), a function in $\cA(\zP E)(U)$ reads $$\sum f(x)\xi^{i_1}\ldots\xi^{i_r}\zh^{j_1}\ldots\zh^{j_s}\zy^{k_1}\ldots\zy^{k_t}\;,$$ where the series is over all $$i_1<\ldots<i_r,\;\,j_1<\ldots<j_s,\;\,k_1\le\ldots\le k_t\;,$$ $$0\le r\le \op{rank}(E_{01}),\;\,0\le s\le \op{rank}(E_{10}),\;\,0\le t\le \infty\;,$$ $$f(x)\in\Ci(U)\;.$$ Therefore, the sheaf $\cA(\zP E)$ is locally canonically isomorphic to $\Ci_{\R^p}[[\xi,\zh,\zy]]$, where $p=\dim M$. Eventually, we associated in a natural way a $\Z_2^2$-supermanifold $(M,\cA(\zP E))$, see (\ref{GradMfdDVB}), to a $\Z_2^2\setminus\{0\}$-graded vector bundle $E$ over $M$. This assignment (\ref{GradMfdDVB}) extends straightforwardly to $\Z_2^n\setminus\{0\}$-graded vector bundles.

\begin{definition} We refer to a $\Z_2^n$-supermanifold $(M,\cA(\zP E))$, which is implemented by a $\Z_2^n\setminus\{0\}$-graded vector bundle $E$ over a manifold $M$, as a \emph{split $\Z_2^n$-supermanifold}.\end{definition}

\begin{remark} Split $\Z_2^n$-supermanifolds and superized $n$-vector bundles, see Section 5 in \cite{CGP14}, are different concepts.\end{remark}

Indeed, let $n=2$ and let $\cE$ be a double vector bundle with side vector bundles $E_{01}\to M$ and $E_{10}\to M$, and core vector bundle $E_{11}\to M$. Then $\cE$ is {\it noncanonically} isomorphic to the double vector bundle $E:=E_{01}\oplus E_{10}\oplus E_{11}$. The {\it space of decompositions} is a nonempty affine space modelled on $\zG(E_{01}^*\otimes E_{10}^*\otimes E_{11})$. 
As for the double vector bundle structure on $E$, recall that the pullback of $E_{10}\oplus E_{11}\to M$ (resp., $E_{01}\oplus E_{11}\to M$) over $E_{01}\to M$ (resp., $E_{10}\to M$) is a vector bundle structure over $E_{01}$ (resp., $E_{10}$) on the manifold $E$. These two bundle structures are compatible and $E$ is actually a double vector bundle. However, the vector bundle structure $E=E_{01}\oplus E_{10}\oplus E_{11}\to M$ is not part of the double vector bundle $E$. The construction of a $\Z_2^2$-supermanifold via superization of $\cE\simeq E$ uses (of course) the double vector bundle structure of $E$ \cite{CGP14}. On the other hand, to build the split $\Z_2^2$-supermanifold associated to $E$, we only needed the vector bundle structure on $E=E_{01}\oplus E_{10}\oplus E_{11}\to M$.

\begin{remark}\label{NonSplitSMfds} Superized $n$-vector bundles are examples of (usually) not canonically split $\Z_2^n$-supermanifolds.\end{remark}

To better understand this claim, think about superization, not intrinsically as above, but, as usual, as the assignment of a degree to each coordinate, provided the coordinate transformations respect this grading and the cocycle condition remains valid for the new noncommuting coordinates (see Remark \ref{CompatSuperCocycle}).

Moreover, look at double vector bundles from the `locally trivial fiber bundle' standpoint, see \cite{Vor12}. A double vector bundle can actually be viewed as a (locally trivial) {\it fiber} bundle $E\to M$ whose standard fiber is a $\Z_2^2\setminus\{0\}$-graded real vector space $V_{01}\oplus V_{10}\oplus V_{11}$ and whose coordinate transformations have the form \be\label{CoordTransDVB}\left\{\begin{array}{l}\za^a=f^{a}_{a'}(x)\xi^{a'}\;,\\ \zb^b=g^{b}_{b'}(x)\zh^{b'}\;,\\ \zg^c=h^{c}_{c'}(x)\zy^{c'}+k^{c}_{a',b'}(x)\xi^{a'}\zh^{b'}\;,\end{array}\right.\ee where the coefficients are smooth in the base coordinates. The $\Z_2^2$-superization of a double vector bundle is a $\Z_2^2$-supermanifold of dimension $\dim M|(\dim V_{01},\dim V_{10},\dim V_{11})$, see \cite{CGP14}, Sections 4.2 and 5.

On the other hand, a $\Z_2^2\setminus\{0\}$-graded vector bundle is a (locally trivial) {\it vector} bundle $E=E_{01}\oplus E_{10}\oplus E_{11}\to M$. It follows that its standard fiber is a $\Z_2^2\setminus\{0\}$-graded real vector space $V_{01}\oplus V_{10}\oplus V_{11}$ and that its coordinate transformations have the form $$\left\{\begin{array}{l}\za^a=f^{a}_{a'}(x)\xi^{a'}\;,\\ \zb^b=g^{b}_{b'}(x)\zh^{b'}\;,\\ \zg^c=h^{c}_{c'}(x)\zy^{c'}\;.\end{array}\right.$$ The superization of a $\Z_2^2\setminus\{0\}$-graded vector bundle is a split $\Z_2^2$-supermanifold of dimension $\dim M|(\dim V_{01},\dim V_{10},\dim V_{11})$. 

Remark \ref{NonSplitSMfds} follows.\medskip

Since possible problems with the cocycle condition after superization are not satisfactorily explained in the literature, let us emphasize that

\begin{remark}\label{CompatSuperCocycle} The $\Z_2^n$-superization of $n$-vector bundles is compatible with the cocycle condition.\end{remark}

Consider the case $n=3$ and assume for simplicity that there exists exactly one formal parameter in each nonzero $\Z_2^3$-degree: \be\label{Order}\xi_{111},\,\xi_{110},\,\xi_{101},\,\xi_{100},\,\xi_{011},\,\xi_{010},\,\xi_{001}\;.\ee

In the coordinate transformations of a 3-vector bundle the variables $\xi_{100},\xi_{010},\xi_{001}$ (resp., $\xi_{110},\xi_{101},\xi_{011}$) transform as $\xi,\zh$ (resp., $\zy$) in (\ref{CoordTransDVB}) above, whereas for $\xi_{111}$, we have \be\label{CoordTransTVD}\xi'_{111}=f(x)\xi_{111}+g(x)\xi_{110}\xi_{001}+h(x)\xi_{101}\xi_{010}+k(x)\xi_{100}\xi_{011}+\ell(x)\xi_{100}\xi_{010}\xi_{001}\;,\ee see \cite{Vor12}.

The point is that, if the cocycle condition holds for the commuting vector bundle coordinates, it must also hold for the supercommuting superized variables (of course, for superized variables we have to fix an order, e.g. the reversed lexicographical order (\ref{Order})). In the case of classical supercommutative variables with commutation rules given by the total degrees, this requirement is satisfied only if one introduces the following sign in the superized coordinate transformation (\ref{CoordTransTVD}): $$\label{CoordTransTVDsign}\xi'_{111}=f(x)\xi_{111}+g(x)\xi_{110}\xi_{001}-h(x)\xi_{101}\xi_{010}+k(x)\xi_{100}\xi_{011}+\ell(x)\xi_{100}\xi_{010}\xi_{001}\;.$$  Indeed, let for instance \be\label{Tr1}\xi''_{111}=\xi'_{111}-\xi'_{101}\xi'_{010}\;\ee and \be\label{Tr2}\xi'_{111}=-2\xi_{101}\xi_{010},\;\,\xi'_{101}=3\xi_{100}\xi_{001},\;\,\xi'_{010}=\xi_{010}\;,\ee where these minus signs appear. Consider now the same transformations without the minus signs and with commuting variables. If the latter satisfy the cocycle condition, i.e. if $$\xi''_{111}=2\xi_{101}\xi_{010}+3\xi_{100}\xi_{010}\xi_{001}\;,$$ then the superized variable $\xi''_{111}$ is given by \be\label{Tr3}\xi''_{111}=-2\xi_{101}\xi_{010}+3\xi_{100}\xi_{010}\xi_{001}\;.\ee It is now easily checked that the transformations (\ref{Tr1}), (\ref{Tr2}) and (\ref{Tr3}), with minus signs and supercommutative variables, satisfy as well the cocycle condition. This is the crucial point in the proof of Theorem 7.1 in \cite{GR09}. The necessity to introduce the minus signs is due to the anticommutation of the classical supervariables $\xi_{010}$ and $\xi_{001}$. In our $\Z_2^n$-graded case, these variables are $\Z_2^3$-commutative, i.e. they commute. This explains why no sign changes are needed in the $\Z_2^n$-case.

\section{Batchelor-Gawedzki theorem}

Even in the case of classical supermanifolds, only a small number of complete proofs of the Batchelor-Gawedzki Theorem can be found in the literature. Below, we give a proof for $\Z_2^n$-supermanifolds that is based on a \v{C}ech cohomology argument \cite{Man}. For sheaf-theoretic issues, we refer the reader to \cite{CGP14}, Section 7.3 and Proof of Proposition 6.7.

\subsection{Cohomological invariant}

Let $\cM=(M,\cA_M)$ be a $\Z_2^n$-supermanifold, $n\ge 1$, let $\ze:\cA_M\to \Ci_M$, $\cJ_M=\ker\ze$, and let $$\cA_M\supset \cJ_M\supset \cJ^2_M\supset\ldots$$ be the decreasing filtration of the structure sheaf by sheaves of $\Z_2^n$-graded ideals. 
To simplify notation, we omit in the sequel subscript $M$. The quotients $\cJ^{k+1}/\cJ^{k+2}$, $k\ge 0$, are locally free sheaves of modules over $\Ci\simeq\cA/\cJ$. 
In particular, $${\cal S}:={\cJ}/{\cJ}^2$$ is a locally free sheaf of $\Z_2^n\setminus\{0\}$-graded $\Ci$-modules (see e.g. \cite{CGP14}, Example 3.2), or, equivalently, a family of $2^n-1$ locally free sheaves of $\Ci$-modules. 
Hence, there exists a $\Z_2^n\setminus\{0\}$-graded vector bundle $E\to M$ such that $${\cal S}\simeq \zG((\zP E)^*)\;.$$ For instance, in the case $n=2$, we get $${\cal S}\simeq\zG(E_{01}[01]^*\oplus E_{10}[10]^*\oplus E_{11}[11]^*)\;.$$ As above, denote by $\odot$ the $\Z_2^n$-graded symmetric tensor product of $\Z_2^n$-graded $\Ci$-modules and of $\Z_2^n$-graded vector bundles. 
Then \be\label{Isos}\zG(\odot^{k+1}(\zP E)^*)\simeq\odot^{k+1}{\cal S}\simeq\cJ^{k+1}/\cJ^{k+2}\;\ee (indeed, the sheaf morphism, which is well-defined on sections by $$\odot^{k+1}\cJ/\cJ^2\ni [s_1]\odot \ldots \odot [s_{k+1}] \mapsto [s_1\cdots s_{k+1}]\in\cJ^{k+1}/\cJ^{k+2}\;,$$ is locally an isomorphism). Our goal is to show that \be\label{Batchelor}\cA(\zP E):=\prod_{k\ge -1}\zG(\odot^{k+1}(\zP E)^*)=\prod_{k\ge -1}\odot^{k+1}{\cal S}\simeq \cA\;\ee as sheaf of $\Z_2^n$-commutative associative unital $\R$-algebras. This $\Z_2^n$-isomorphism implies indeed the

\begin{thm}[Batchelor-Gawedzki Theorem for $\Z_2^n$-Supermanifolds] Any smooth $\Z_2^n$-super-manifold is (noncanonically) isomorphic to a split $\Z_2^n$-supermanifold.\end{thm}

It is clear that locally the two considered sheaves (see (\ref{Batchelor})) coincide. To prove that they are isomorphic, we will build a morphism $\prod_{k\ge -1}\odot^{k+1}{\cal S}\to \cA$ of sheaves of $\Z_2^n$-commutative associative unital $\R$-algebras. The idea is to extend a morphism ${\cal S}\to \cA$, or $\cJ/\cJ^2\to \cJ$. The latter will be obtained as a splitting of the sequence $0\to \cJ^2\to \cJ\to \cJ/\cJ^2\to 0$. One of the problems to solve is to show that this sequence can be viewed as a sequence of sheaves of $\Ci$-modules. Therefore, we need an embedding $\Ci\to \cA$.

\subsection{Projection of $\cM$ onto $M$}

Let $\cM=(M,\cA)$ be a $\Z_2^n$-supermanifold of dimension $p|\mathbf{q}$ and let $0\to\cJ\to\cA\stackrel{\ze}{\to}\Ci\to 0$ be the corresponding basic short exact sequence ({\small SES}). We will embed (noncanonically) $\Ci$ into $\cA$, i.e. construct a morphism $\zf:\Ci\to \cA$ of sheaves of $\Z_2^n$-commutative associative unital $\R$-algebras, such that $\ze\circ\zf=\op{id}$. In the case of $\N$-manifolds this embedding is canonical, what makes the proof of the corresponding Batchelor-Gawedzki theorem much simpler.\medskip

We build $\zf$ as the limit of an $\N$-indexed sequence of morphisms $\zf_{k}:\Ci\to \cA/\cJ^{k+1}$ of sheaves of $\Z_2^n$-commutative associative unital $\R$-algebras:

\[
\begin{tikzpicture}
\matrix(m)[matrix of math nodes, column sep=2em, row sep=4em]
{
&&&\Ci&&&\\
&&&\cA=\varprojlim_k\cA/\cJ^k&&&\\
\ldots&&{\cA}/{\cJ^{k+1}}&&{\cA}/{\cJ^{k+2}}&&\ldots\\
};
\path[->]
(m-3-7)edge node[auto]{$ $} (m-3-5)
(m-3-5)edge node[auto]{$f_{k,k+1} $} (m-3-3)
(m-3-3)edge node[auto]{$ $} (m-3-1)
(m-2-4)edge node[right]{$\pi_k$} (m-3-3)
            edge node[left]{$\pi_{k+1}$} (m-3-5)
(m-1-4)edge[bend right=20] node[left]{$\zf_{k}$} (m-3-3)
            edge[bend left=20] node[right]{$\zf_{k+1}$} (m-3-5)
(m-1-4)edge node[auto]{$\zf$} (m-2-4)
;
\end{tikzpicture}
\]

\noindent The sequence $\zf_k$ will be obtained by induction on $k$, starting from $\zf_0=\id$: we assume that we already got $\zf_{i+1}$ as an extension of $\zf_i$ for $0\le i\le k-1$, and we aim at extending $\zf_k:\Ci\to \cA/\cJ^{k+1}$ to 
$$\zf_{k+1}:\Ci\to \cA/\cJ^{k+2}\;.$$ 
The word `extension' is used here in the sense that 
$$\label{Ext}f_{k,k+1}\circ\zf_{k+1}=\zf_{k}\;.$$ 
For any open subset $\zW\subset M$, we build extensions $\zf_{k+1,\zW}:\Ci(\zW)\to \cA(\zW)/\cJ^{k+2}(\zW)$ of $\zf_{k,\zW}$, via a consistent construction of extensions of the $\zf_{k,U}$ by local (in the sense of presheaf morphisms) degree zero unital $\R$-algebra morphisms 
$$\label{LocExt}\zf_{k+1,U}:\Ci(U)\to \cA(U)/\cJ^{k+2}(U)\simeq\Ci(U)[[\xi^1,\ldots,\xi^q]]_{\le k+1}$$
over a cover $\cal U$ of $\zW$ by $\Z_2^n$-superchart domains $U$. Here subscript $\le k+1$ means that we confine ourselves to `series' whose terms contain at most $k+1$ formal parameters. Further, `consistent' means that, if $U,V$ are two domains of the cover, we must have 
$$\label{Consistency1}\zf_{k+1,U}|_{U\cap V}=\zf_{k+1,V}|_{U\cap V}\;.$$
(Indeed, if $f_\zW\in\Ci(\zW)$, the sections $\zf_{k+1,U}(f_\zW|_U)\in \cA(U)/\cJ^{k+2}(U)$, $U\in\cal U$, define a unique section $\zf_{k+1,\zW}(f_\zW)\in \cA(\zW)/\cJ^{k+2}(\zW)$, if their restrictions to the intersections $U\cap V$ coincide, i.e. if 
$$\label{Consistency2}\zf_{k+1,U}(f_\zW|_U)=\zf_{k+1,V}(f_\zW|_V)$$
on $U\cap V$.) 

\begin{lem}
Over any $\Z_2^n$-chart domain $U$, there exists an extension $\zf_{k+1,U}:\Ci(U)\to \cO(U)_{\le k+1}:=\Ci(U)[[\xi^1,\ldots,\xi^q]]_{\le k+1}$ of $\zf_{k,U}$ as local degree 0 unital $\R$-algebra morphism.
\end{lem}

\begin{proof} We look for an extension $\zf_{k+1,U}$ of the local degree 0 unital $\R$-algebra morphism $\zf_{k,U}:\Ci(U)\to \cO(U)_{\le k}\subset \cO(U)$ (where the latter is built step by step as an extension of $\zf_0=\op{id}$). Denote by $x=(x^1,\ldots,x^p)$ the base coordinates in $U$. The `pullbacks' $$\zf_{k,U}(x^i)=x^i+\sum_{1\le|\zm|\le k}f^i_\zm(x)\xi^\zm\in\cO(U)$$ uniquely define a local degree 0 unital $\R$-algebra morphism $\overline{\zf}_{k,U}:\Ci(U)\to \cO(U)$, see Theorem \ref{FundaTheoSuperm}. 
Since the algebra structure in $\cO(U)_{\le k}$ is given by the multiplication of $\cO(U)$ truncated at order $k$, it is easily seen that the restriction $\overline{\zf}_{k,U}|_{\le k}:\Ci(U)\to \cO(U)_{\le k}$ is still a local degree 0 unital $\R$-algebra morphism. 
For the same reason, the morphisms $\zf_{k,U}$ and $\overline{\zf}_{k,U}|_{\le k}$ coincide on polynomial functions $P(x)\in\Ci(U)$. We will actually prove that these morphisms coincide on all functions $f(x)\in\Ci(U)$. Then $\overline{\zf}_{k,U}|_{\le k+1}$ is the searched extension $\zf_{k+1,U}$.\medskip

Let now $x_0\in U$ and denote by $$\frak{m}_{x_0}=\{[g]_{x_0}: g(x_0)=0\}\quad\text{and}\quad\frak{m}'_{x_0}=\{[h]_{x_0}: (\ze h)(x_0)=0\}$$ the unique maximal homogeneous ideal of $\Ci_{x_0}$ and $\cO_{x_0}$, respectively. The morphism $\overline{\zf}_{k,U}$ (resp., $\overline{\zf}_{k,U}|_{\le k}$, $\zf_{k,U}$) is a local (in the sense of presheaf morphism) degree zero unital $\R$-algebra morphism (resp., are local degree zero $\R$-linear maps) $\Ci(U)\to\cO(U)$ and thus defines an algebra morphism (resp., linear maps) $\overline{\zf}_{k,x_0}$ (resp., $\overline{\zf}_{k,x_0}|_{\le k}$, $\zf_{k,x_0}$) between $\Ci_{x_0}$ and $\cO_{x_0}$. The maps $\overline{\zf}_{k,x_0}|_{\le k}$ and $\zf_{k,x_0}$ send $\frak{m}^\ell_{x_0}$ into $\frak{m'}^\ell_{x_0}$, $\ell\ge 1$.

Indeed, if $[g]_{x_0}\in\frak{m}_{x_0}^\ell$, then $$[\overline{\zf}_{k,U}(g)]_{x_0}=\overline{\zf}_{k,x_0}[g]_{x_0}\in\frak{m}'^\ell_{x_0}\;,$$ so $$\overline{\zf}_{k,x_0}|_{\le k}[g]_{x_0}=[\overline{\zf}_{k,U}(g)|_{\le k}]_{x_0}\in\frak{m}'^\ell_{x_0}\;,$$ in view of Propositions \ref{Claim1} and \ref{Claim4}.

Since the target of $\zf_{k,U}$ is truncated and thus not a `$\Z_2^n$-superdomain', some care is required. Note first that, if $[g]_{x_0}\in \frak{m}_{x_0}$, then $\ze(\zf_{k,U}g)(x_0)=g(x_0)=0$, so that $\zf_{k,x_0}[g]_{x_0}\in\frak{m}'_{x_0}$. Moreover, if $[g_1]_{x_0},\ldots,[g_\ell]_{x_0}\in\frak{m}_{x_0}$, then $$[\zf_{k,U}(g_1)\ldots\zf_{k,U}(g_\ell)]_{x_0}=[\zf_{k,U}(g_1)]_{x_0}\ldots[\zf_{k,U}(g_\ell)]_{x_0}=\zf_{k,x_0}[g_1]_{x_0}\ldots\zf_{k,x_0}[g_\ell]_{x_0}\in\frak{m}'^\ell_{x_0}\;,$$ so that $$\zf_{k,x_0}\lp[g_1]_{x_0}\ldots[g_\ell]_{x_0}\rp=\zf_{k,x_0}[g_1\ldots g_\ell]_{x_0}=[\lp\zf_{k,U}(g_1)\ldots\zf_{k,U}(g_\ell)\rp|_{\le k}]_{x_0}\in\frak{m}'^\ell_{x_0}\;.$$

Consider finally $f(x)\in\Ci(U)$ and $x_0\in U$, as well as the `series' $$\zf_{k,U}(f)-\overline{\zf}_{k,U}(f)|_{\le k}\in\cO(U)_{\le k}\;.$$ Let $\ell>k$. Theorem \ref{Claim3} implies that there is a polynomial $P(x)$ such that $[f]_{x_0}-[P]_{x_0}\in\frak{m}_{x_0}^\ell.$ It follows that $$[\zf_{k,U}(f)-\overline{\zf}_{k,U}(f)|_{\le k}]_{x_0}= \zf_{k,x_0}\lp[f]_{x_0}-[P]_{x_0}\rp-\overline{\zf}_{k,x_0}|_{\le k}\lp[f]_{x_0}-[P]_{x_0}\rp\in\frak{m}'^\ell_{x_0}\;.$$ Hence, all the coefficients of $\zf_{k,U}(f)-\overline{\zf}_{k,U}(f)|_{\le k}$ vanish at $x_0$, see Proposition \ref{Claim4}, for all $x_0\in U$, and all functions $f(x)\in\Ci(U)$.\end{proof}

To finalize the construction of the sheaf morphism $\zf:\Ci\to \cA$, it now suffices to solve the consistency problem. Let $U$ and $V$ be $\Z_2^n$-chart domains and let $\zf_{k+1,U}$ and $\zf_{k+1,V}$ be extensions of $\zf_{k,U}$ and $\zf_{k,V}$, respectively, which exist according to the preceding lemma. 
The difference
$$\label{Consistency3}\zw_{{k+1},UV}(f):=\zf_{k+1,U}|_{U\cap V}(f)-\zf_{k+1,V}|_{U\cap V}(f)\in\cO(U\cap V)_{\le k+1}\;,$$ 
$f\in\Ci(U\cap V)$,
defines a derivation $$\zw_{k+1,UV}:\Ci(U\cap V)\to\cO(U\cap V)_{=k+1}\;.$$

Indeed, since $\zf_{k+1,U}=\overline{\zf}_{k,U}|_{\le k+1}=\zf_{k,U}+\overline{\zf}_{k,U}|_{= k+1}$, we have $$\zw_{k+1,UV}(f)=\zf_{k,U}|_{U\cap V}(f)+\overline{\zf}_{k,U}|_{=k+1}|_{U\cap V}(f)-\zf_{k,V}|_{U\cap V}(f)-\overline{\zf}_{k,V}|_{=k+1}|_{U\cap V}(f)\;,$$ where the restrictions to $U\cap V$ are valued in $\cO(U\cap V)|_{\le k+1}$, i.e., after coordinate transformation we omit the terms of order $> k+1$. Note now that in a coordinate transformation the order cannot decrease, so that the second and fourth terms of the {\small RHS} contain only terms of order $k+1$. The same holds for the difference of the first and third terms. Indeed, since the $\zf_{k,U}$ have already been constructed consistently, they coincide on intersections up to order $k$: the remaining terms are of order $k+1$.

As for the derivation property, start from

$$\zf_{k+1,U}|_{U\cap V}(fg)=\zf_{k+1,V}|_{U\cap V}(fg)+\zw_{k+1,UV}(fg)\,.$$
The left hand side equals
\be\zf_{k+1,U}|_{U\cap V}(f)\cdot\zf_{k+1,U}|_{U\cap V}(g)=
\zf_{k+1,V}|_{U\cap V}(f)\cdot\zf_{k+1,V}|_{U\cap V}(g)+
f\cdot\zw_{k+1,UV}(g)+\zw_{k+1,UV}(f)\cdot g\;.
\ee
Indeed, the products are products in $\cO(U\cap V)$ truncated at order $k+1$, i.e. we omit the terms of order $\ge k+2$. Since
$$\zf_{k+1,V}|_{U\cap V}(f)\cdot\zf_{k+1,V}|_{U\cap V}(g)=\zf_{k+1,V}|_{U\cap V}(fg)\;,
$$
we finally get
\be \zw_{k+1,UV}(fg)=\zw_{k+1,UV}(f)\cdot g+f\cdot\zw_{k+1,UV}(g)\;.
\ee
In view of (\ref{Isos}), the map $\zw_{k+1,UV}$ is a derivation $$\zw_{k+1,UV}:\Ci(U\cap V)\to (\cJ^{k+1}(U\cap V))^0/(\cJ^{k+2}(U\cap V))^0\simeq \zG(U\cap V,(\odot^{k+1}(\zP E)^*)^{0})\;,$$ i.e. it is a vector field valued in symmetric $(k+1)$-tensors: $$\zw_{k+1,UV}\in \zG(U\cap V,TM\otimes (\odot^{k+1}(\zP E)^*)^{0})\;.$$ This \v{C}ech 1-cochain $\zw_{k+1}$ is obviously a 1-cocycle. However, as well-known, the existence of a partition of unity in $M$ implies that $\check{H}^{\bullet\ge 1}(M,\cE)=0$, for any locally free sheaf $\cE$ over $M$. Hence, there exists a 0-cochain $\zh_{k+1}$, i.e. a family $\zh_{k+1,U}\in \zG(U,TM\otimes(\odot^{k+1}(\zP E)^*)^{0})$, or, still, a family of derivations $$\zh_{k+1,U}:\Ci(U)\to \zG(U,(\odot^{k+1}(\zP E)^*)^{0})\simeq (\cJ^{k+1}(U))^0/(\cJ^{k+2}(U))^0\simeq\cO^0(U)_{= k+1}\;,$$ such that $$\zf_{k+1,U}|_{U\cap V}-\zf_{k+1,V}|_{U\cap V}=\zw_{k+1,UV}=\zh_{k+1,V}|_{U\cap V}-\zh_{k+1,U}|_{U\cap V}\;.$$ It is now easily checked that, since the multiplication in $\cO(U)_{\le k+1}$ is the truncation of the standard multiplication in $\cO(U)$, the sum $\zf'_{k+1,U}:=\zf_{k+1,U}+\zh_{k+1,U}:\Ci(U)\to\cO(U)_{\le k+1}$ is a local degree 0 unital $\R$-algebra morphism, which satisfies the consistency condition and extends $\zf_{k,U}$. This proves the existence of the searched morphism $\zf:\Ci\to \cA$ of sheaves of $\Z_2^n$-commutative associative unital $\R$-algebras.\medskip

More precisely, we constructed sheaf morphism $$\zf:\Ci\to \cA\;,$$ which splits the {\small SES} $$0\to \cJ\longrightarrow \cA\stackrel{\ze}{\longrightarrow}\Ci\to 0$$ of $\Z_2^n$-commutative associative unital $\R$-algebras, i.e. which satisfies $\ze\circ\zf=\op{id}$. Indeed, for any open subset $V\subset M$ and any $f\in\Ci(V)$, we have, on an open cover by $\Z_2^n$-chart domains $U\subset V$, $$(\ze_V\zf_V f)|_U=\ze_U\zf_U(f|_U)=(\op{id}_Vf)|_U\;.$$ Hence, the

\begin{thm}
For any $\Z_2^n$-supermanifold $(M,\cA_M)$, the short exact sequence $$0\to {\cal J}_M\to \cA_M\stackrel{\ze}{\to }\Ci_M\to 0$$ of sheaves of $\Z_2^n$-commutative associative unital $\R$-algebras is noncanonically split.
\end{thm}

\subsection{Algebra morphisms}

Let us recall that a {\small SES} $0\to W\to V\to U\to 0$ of smooth vector bundles over a same smooth manifold $M$ (and vector bundle maps that cover identity) is always split -- essentially because there exist smooth partitions of unity in $M$. Indeed, we can endow $V$ with a Riemannian metric (we glue local metrics by means of a partition of unity). This (global) metric induces a metric on each fibre $V_m$ and $V_m$ splits as $V_m = W_m\oplus W^\perp_m$, with self-explaining notation. We thus get a global splitting $V=W\oplus W^\perp$ as the metric is smooth. Of course, $W^\perp\simeq U$, as both bundles are isomorphic to the quotient bundle $V/W$.\medskip

It follows that a {\small SES} of locally free sheaves of $\Ci_M$-modules, where $M$ is a smooth manifold, is always split.\medskip

Of course, due to the embedding $\zf:\Ci\to \cA$, any $\cA$-module inherits a $\Ci$-module structure. As aforementioned, we need a splitting of the {\small SES} $$0\to \cJ^2\to \cJ\to {\cal S}=\cJ/\cJ^2\to 0$$ of sheaves of $\Ci_M$-modules. Since $\cJ^2$ and $\cJ$ are not locally free, we consider the {\small SES}s $$0\to \cJ^2/\cJ^k\stackrel{i_k}{\longrightarrow} \cJ/\cJ^k\stackrel{p_k}{\longrightarrow} {\cal S}=\cJ/\cJ^2\to 0\;,$$ $k\ge 2$, of locally free sheaves of $\Ci$-modules and denote by $\zF_k$ a splitting: $p_k\circ\zF_k=\op{id}_{\cal S}$. Since the category of sheaves of $\Ci$-modules is Abelian and since in any Abelian category the inverse limit functor is left exact, we get the exact sequence of sheaves of $\Ci$-modules $$0\to \lim \cJ^2/\cJ^k\stackrel{i=\lim i_k}{\longrightarrow} \lim \cJ/\cJ^k\stackrel{p=\lim p_k}{\longrightarrow} {\cal S}=\cJ/\cJ^2\;,$$ or, still, $$0\to \cJ^2\stackrel{i}{\longrightarrow} \cJ\stackrel{p}{\longrightarrow} {\cal S}=\cJ/\cJ^2\to 0\;,$$ where exactness at the last spot is obvious. When setting $\zF=\lim\zF_k$, we get $\zF:{\cal S}\to \cJ\subset\cA$ such that $$p\circ\zF =(\lim p_k)\circ(\lim \zF_k)=\lim (p_k\circ\zF_k)=\op{id}_{\cal S}\;.$$ 
Note that we actually deal with sheaves of $\Z_2^n$-graded $\Ci$-modules and corresponding sheaf morphisms. Therefore, the splitting is in fact a family of degree zero $\Ci$-linear maps $\zF_U:{\cal S}(U)\to \cJ(U)\subset\cA(U)$, $U\subset M$ (that commute with restrictions).\medskip

We now extend $\zF$ to a morphism $\cA(\zP E)=\prod_{k\ge 0}\odot^k{\cal S}\to \cA$ of sheaves of $\Z_2^n$-commutative associative unital $\R$-algebras. Since such a morphism is made of a family of degree 0 unital $\R$-algebra morphisms between algebras of sections (that commute with restrictions), we deal in the sequel mainly with section spaces, but no notational difference will be made between the sheaves and their spaces of sections. As announced, we will show that the degree zero $\Ci$-linear map $\zF:{\cal S}\to \cJ\subset\cA$ extends to the needed degree 0 unital $\R$-algebra morphism -- also denoted by $\zF$. Indeed, define $\zF$ first for each degree $k\ge 0$. For $k=0$, i.e. on $\Ci$, set $\zF:=\zf:\Ci\to \cA$, where $\zf$ is the above-constructed degree preserving unital algebra morphism. For $k=1$, i.e. on ${\cal S}$, the map $\zF$ is the inducing map $\zF:{\cal S}\to \cJ\subset\cA$. On $\odot^{k\ge 2}{\cal S}$, 
we set, for any $\psi_1,\ldots,\psi_k\in{\cal S}$, $$\zF(\psi_1\odot\ldots\odot\psi_k):=\zF(\psi_1)\cdot\ldots\cdot\zF(\psi_k)\in \cJ^k\subset\cA\;.$$ This extension is well-defined since the {\small RHS} is $\Z_2^n$-commutative and $\Ci$-multilinear. Indeed, the multiplication in $\cA$ has these properties; in particular, if $f\in\Ci$ and $s',s''\in\cA$, we have $$s'\cdot(f s'')=s'\cdot(\zf(f)\cdot s'')=\zf(f)\cdot(s'\cdot s'')=f(s'\cdot s'')\;.$$ Eventually, for $\sum_{k=0}^\infty\Psi_k\in\cA(\zP E)=\prod_{k\ge 0}\odot^k{\cal S},$ we set $$\zF(\sum_{k=0}^\infty\Psi_k):=\sum_{k=0}^\infty\zF(\Psi_k)\;,$$ where the {\small RHS} is a Cauchy sequence of partial sums in $\cA$ with respect to the filtration induced by $\cJ$. Since $\cA$ is Hausdorff-complete, this {\small RHS} sequence has a unique limit $\sum_{k=0}^\infty\zF(\Psi_k)\in\cA$. \smallskip

This map $\zF:\cA(\zP E)\to \cA$ respects the degrees and the units, and is an $\R$-algebra morphism. Indeed, note first that if $\sum_{k=0}^ns'_k\to s'$ and $\sum_{\ell=0}^ns''_\ell\to s''$ are two Cauchy sequences in $\cA$, then \be\label{Conv1}(\sum_{k=0}^ns'_k)\cdot(\sum_{\ell=0}^ns''_\ell)\to s'\cdot s''\;,\ee since $$(\sum_{k=0}^ns'_k)\cdot(\sum_{\ell=0}^ns''_\ell)-s'\cdot s''=(\sum_{k=0}^ns'_k-s')\cdot\sum_{\ell=0}^ns''_\ell+s'\cdot(\sum_{\ell=0}^ns''_\ell-s'')\in \cJ^{n+1}\;.$$ Further, we have \be\label{Conv2}\zF(\sum_{k=0}^\infty\Psi_k'\odot\sum_{\ell=0}^\infty\Psi''_\ell)=\sum_{n=0}^\infty\sum_{k+\ell=n}\zF(\Psi'_k\odot\Psi''_\ell)=\sum_{n=0}^\infty\sum_{k+\ell=n}\zF(\Psi'_k)\cdot\zF(\Psi''_\ell)\;.\ee It follows that, for any $r$, in $$\zF(\sum_{k=0}^\infty\Psi_k'\odot\sum_{\ell=0}^\infty\Psi''_\ell)-\sum_{n=0}^r\sum_{k+\ell=n}\zF(\Psi'_k)\cdot\zF(\Psi''_\ell)\,+$$ $$\sum_{n=0}^r\sum_{k+\ell=n}\zF(\Psi'_k)\cdot\zF(\Psi''_\ell)-\sum_{k=0}^r\zF(\Psi'_k)\cdot \sum_{\ell=0}^r\zF(\Psi''_\ell)\,+$$ $$\sum_{k=0}^r\zF(\Psi'_k)\cdot \sum_{\ell=0}^r\zF(\Psi''_\ell)-\zF(\sum_{k=0}^\infty\Psi_k')\cdot\zF(\sum_{\ell=0}^\infty\Psi''_\ell)\;,$$ the first difference (see (\ref{Conv2})) and the third difference (see (\ref{Conv1})) belong to $\cJ^{r+1}.$ As for the second difference, observe that its second term contains all the products of the first, but also additional terms, which however belong all to $\cJ^{r+1}$. As $\cap_s\cJ^s=\{0\}$, it follows that $$\zF(\sum_{k=0}^\infty\Psi_k'\odot\sum_{\ell=0}^\infty\Psi''_\ell)=\zF(\sum_{k=0}^\infty\Psi_k')\cdot\zF(\sum_{\ell=0}^\infty\Psi''_\ell)\;.$$

Actually the just constructed sheaf morphism is locally an isomorphism, what completes the proof of the theorem.

\section{Acknowledgements}

T. Covolo thanks the Luxembourgian NRF for support via AFR grant 2010-1, 786207. The research of J. Grabowski has been founded by the  Polish National Science Centre grant under the contract number DEC-2012/06/A/ST1/00256. The research of N. Poncin has been supported by the Grant GeoAlgPhys 2011-2014 awarded by the University of Luxembourg. Moreover, the authors are grateful to Sophie Morier-Genoud and Valentin Ovsienko, whose work on $\Z_2^n$-gradings \cite{MGO1}, \cite{MGO2} was the starting point of this series of papers on $\Z_2^n$-Supergeometry. They thank Dimitry Leites for his suggestions and valuable support. N. Poncin appreciated Pierre Schapira's explanations on sheaves.

\small{\vskip1cm

\noindent Tiffany COVOLO\\University of
Luxembourg\\ Campus Kirchberg, Mathematics Research Unit\\ 6, rue Richard Coudenhove-Kalergi, L-1359 Luxembourg
City, Grand-Duchy of Luxembourg
\\Email: tiffany.covolo@uni.lu \\

\noindent Janusz GRABOWSKI\\ Polish Academy of Sciences\\ Institute of
Mathematics\\ \'Sniadeckich 8, P.O. Box 21, 00-656 Warsaw,
Poland\\Email: jagrab@impan.pl \\

\noindent Norbert PONCIN\\University of Luxembourg\\
Campus Kirchberg, Mathematics Research Unit\\ 6, rue Richard Coudenhove-Kalergi, L-1359 Luxembourg
City, Grand-Duchy of Luxembourg\\Email: norbert.poncin@uni.lu}

\end{document}